\documentclass[12pt]{amsart}
\usepackage{epsfig,color}
\usepackage{blindtext}
\usepackage{hyperref}
\usepackage{graphicx}
\usepackage{enumitem} 
\usepackage{url}
\usepackage{amssymb}
\usepackage{graphicx,import}

\theoremstyle{definition}

\newtheorem{theorem}{Theorem}[section]
\newtheorem{definition}[theorem]{Definition}

\newtheorem{proposition}[theorem]{Proposition}
\newtheorem{lemma}[theorem]{Lemma}
\newtheorem{remark}[theorem]{Remark}
\newtheorem{corollary}[theorem]{Corollary}

\newtheorem*{remark*}{Remark}

\numberwithin{equation}{section}

\newcommand{\dv}{\text{div}}

\newcommand{\mc}{\mathcal}
\newcommand{\mb}{\mathbb}

\newcommand{\eps}{\varepsilon}

\DeclareMathOperator{\area}{Area}

\title[Compactness and Rigidity of $\lambda$-Surfaces]{Compactness and Rigidity of $\lambda$-Surfaces}

\date{\today}

\author{Ao Sun}
\address{Department of Mathematics, Massachusetts Institute of Technology, Cambridge, MA 02139, USA}
\email{aosun@mit.edu}

\begin{document}

\begin{abstract}
In this paper we develop the compactness theorem for $\lambda$-surface in $\mb R^3$ with uniform $\lambda$, genus, and area growth. This theorem can be viewed as a generalization of Colding-Minicozzi's compactness theorem for self-shrinkers in $\mb R^3$. As an application of this compactness theorem, we prove a rigidity theorem for convex $\lambda$-surfaces.
\end{abstract}

\maketitle

\section{Introduction}
A surface $\Sigma\subset\mb R^3$ is said to be a {\bf $\lambda$-surface} if it is a critical point for Gaussian area enclosing fixed Gaussian volume. Recall the Gaussian area and Gaussian volume are area and volume in $\mathbb{R}^3$ with Gaussian measure density $e^{-\frac{\vert x\vert^2}{4}}$. $\lambda$-surfaces naturally arise in geometry and probability theory. When $\lambda=0$, $\lambda$-surfaces are also called self-shrinkers, which are the models of the singular points of mean curvature flow.

A $\lambda$-surface satisfies the equation
\[H=\frac{\langle x,\mathbf n\rangle}{2}+\lambda\]
where $H=\dv~\mathbf{ n}$ is the mean curvature of $\Sigma$, $x$ is the position vector in $\mb R^3$, and $\mathbf{n}$ is the unit normal vector of $\Sigma$.

In this paper, we prove a compactness theorem for $\lambda$-surfaces.
\begin{theorem}\label{mainthm}
Given an integer $g\geq 0$ and $\Lambda\geq0$. Suppose $\Sigma_i\subset\mb R^3$ is a sequence of smooth complete embedded $\lambda_i$-surfaces with genus at most $g$ and $\partial\Sigma_i=\emptyset$, satisfying:
\begin{enumerate}
\item $\area(B_{R}(x_0)\cap\Sigma_i)\leq C(x_0)R^2$ for all $x_0\in\mb R^3$ with a constant $C$ depending on $x_0$ and all $R>0$,
\item $\vert \lambda_i\vert\leq\Lambda$.
\end{enumerate}
Then 

\begin{enumerate}
\item either: there exists a complete smooth self-touching immersed $\lambda$-surface $\Sigma$ with local finitely many neck pinching points, such that a subsequence of $\Sigma_i$ converges to $\Sigma$ in $C^k$ topology in any compact subset of $\mb R^3$ which does not contain those neck pinching points, for any $k\geq 2$.
\item or: there exists a complete smooth self-shrinker $\Sigma$ and a subsequence of $\Sigma_i$ converges to $\Sigma$ with multiplicity $2$.
\end{enumerate}
\end{theorem}

Here $\Sigma$ is assumed to be homeomorphic to a closed surface with finitely many disjoint closed disks removed. The genus is defined to be the genus of the closed surface.  Locally finite means finite inside any ball $B_R$ for any $R>0$. Self-touching means that at each non-embedded points $p$ of $\Sigma$, $\Sigma$ can be written as the union of two graphs over the tangent plane at $p$, and the function $u_1,u_2$ defining the graphs satisfies $u_1\geq u_2$. See picture.

\begin{figure}\label{figure1}
\centering
\scalebox{1}{
\begingroup%
  \makeatletter%
  \providecommand\color[2][]{%
    \errmessage{(Inkscape) Color is used for the text in Inkscape, but the package 'color.sty' is not loaded}%
    \renewcommand\color[2][]{}%
  }%
  \providecommand\transparent[1]{%
    \errmessage{(Inkscape) Transparency is used (non-zero) for the text in Inkscape, but the package 'transparent.sty' is not loaded}%
    \renewcommand\transparent[1]{}%
  }%
  \providecommand\rotatebox[2]{#2}%
  \newcommand*\fsize{\dimexpr\f@size pt\relax}%
  \newcommand*\lineheight[1]{\fontsize{\fsize}{#1\fsize}\selectfont}%
  \ifx\svgwidth\undefined%
    \setlength{\unitlength}{212.5984252bp}%
    \ifx\svgscale\undefined%
      \relax%
    \else%
      \setlength{\unitlength}{\unitlength * \real{\svgscale}}%
    \fi%
  \else%
    \setlength{\unitlength}{\svgwidth}%
  \fi%
  \global\let\svgwidth\undefined%
  \global\let\svgscale\undefined%
  \makeatother%
  \begin{picture}(1,0.69333333)%
    \lineheight{1}%
    \setlength\tabcolsep{0pt}%
    \put(0,0){\includegraphics[width=\unitlength,page=1]{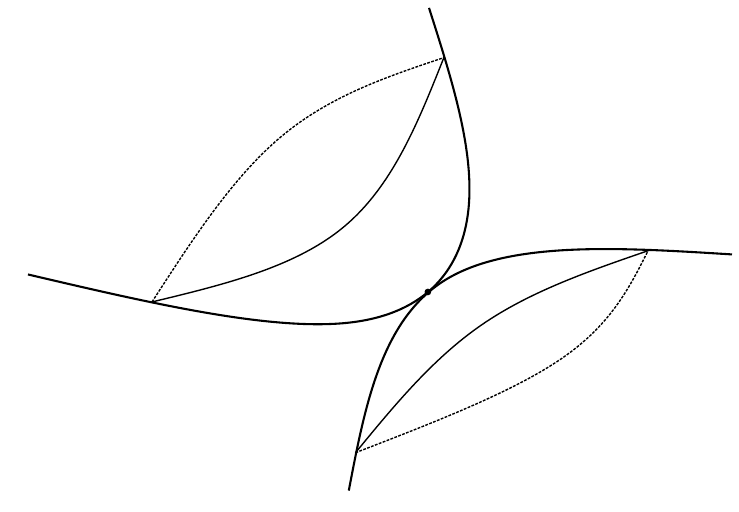}}%
    \put(0.01996682,0.13003215){\color[rgb]{0,0,0}\makebox(0,0)[lt]{\lineheight{1.25}\smash{\begin{tabular}[t]{l}Touching point\end{tabular}}}}%
  \end{picture}%
\endgroup%
}
\end{figure}    

Neck pinching points are special touching points which behaves badly in the convergence process. We will give precise definition in Section \ref{Section:Compactness}. Intuitively they are touching points generated by a neck pinching process. In particular, if the limit surface $\Sigma$ does not contain touching points, then the convergence is smooth in any compact subset in $\mb R^3$.

This compactness theorem generalizes \cite{colding2012smooth}, in which Colding-Minicozzi proved the compactness result for self-shrinkers, which is our main theorem with $\lambda=0$. \cite{colding2012smooth} played an important role in the study of self-shrinkers, especially for the rigidity of self-shrinkers. Here, we can also use the compactness theorem of $\lambda$-surfaces to study the rigidity of $\lambda$-surfaces. In particular, we have the following rigidity theorem:

\begin{theorem}\label{rigiditythm}
There exists $\delta_D<0$, such that for any $\lambda\in(\delta_D,+\infty)$, any convex $\lambda$-surface with diameter less than $D$ has to be sphere $S(r)$ with radius $r=\sqrt{\lambda^2+4}-\lambda$.
\end{theorem} 

This theorem was previously only known for $\lambda\geq0$ case. In \cite{colding2012generic}, Colding and Minicozzi showed that mean convex self-shrinkers in $\mb R^{n+1}$ for all $n\geq2$ can only be spheres and generalized cylinders. Later in \cite{heilman2017symmetric} Heilman generalized this kind of rigidity to $\lambda$-hyersurfaces in $\mb R^{n+1}$ for $\lambda>0$. 

The case $\lambda<0$ is much more complicated. In \cite{chang20171} Chang constructed $ 1$ dimensional $\lambda$-curves in $\mb R^2$ which are not circles. These surprising examples show that $\lambda$-surfaces behave very differently when $\lambda<0$ compare to $\lambda>0$. New techniques must be introduced to study this phenomenon. We imagine these examples may carry meaningful information in probability theory.

\subsection{$\lambda$-hypersurfaces}
From a variational point of view, a $\lambda$-hypersurface is the critical point of the Gaussian area

\begin{equation}
\mathcal{F}(\Sigma)=\int_\Sigma e^{-\frac{\vert x\vert^2}{4}}d\mu
\end{equation}

with variational field $V=f\mathbf{n}$ preserving the volume of the domain enclosed by $\Sigma$, i.e.

\[\int_\Sigma f e^{-\frac{\vert x\vert^2}{4}}d\mu=0.\]

Self-shrinkers are hypersurfaces in $\mathbb{R}^{n+1}$ satisfying the $\lambda$-hypersurface equation with $\lambda=0$. Self-shrinkers are the models for singularities of mean curvature flows, and have been studied in many papers. See \cite{huisken1990asymptotic}, \cite{colding2012generic}. 

In probability theory, $\lambda$-hypersurfaces are related to the isoperimetric problem in Gaussian space. This space plays a central role in probability, and the isoperimetric problem in Gaussian space is of long time interest in probability, see for example \cite{isaksson2012maximally}, \cite{mossel2015robust}. If the boundary $\Sigma$ of the best solution to isoperimetric problem is smooth, then it must be a minimizer of the $\mathcal{F}$ functional under variations preserving the volume enclosed by $\Sigma$, i.e. a $\lambda$-hypersurface. So one approach to the isoperimetric problem in Gaussian space is to study $\lambda$-hypersurfaces. 

In 70's, Sudakov and Cirel'son \cite{sudakov1978extremal} showed that the most optimal solutions to the isoperimetric problem in Gaussian space are hyperplanes. Later McGonagle and Ross \cite{mcgonagle2015hyperplane} also proved this result by techniques in geometric analysis. In 2001, Barthe \cite{barthe2001isoperimetric} studied the Gaussian isoperimetric problem when the set is required to be symmetric. He suggested that symmetric two planes could be the most optimal case in symmetric setting. Some other research also suggested that the most optimal case could be a sphere, for details see the introduction of \cite{heilman2017symmetric}. We hope our study could play a role in this probability problem.

\subsection{Compactness Theorem}
The compactness theorem for minimal surfaces in $3$ dimensional manifolds was first developed by Choi and Schoen in \cite{choi1985space}. Later their result were generalized to many different situations. For examples, White generalized the compactness theorem to stationary surfaces of parametric elliptic functional in \cite{white1987curvature}, Colding and Minicozzi generalized the compactness theorem to self-shrinkers in \cite{colding2012smooth}, and later Ding and Xin weakened the requirements in \cite{ding2013volume}. There are also other generalizations, see \cite{cheng2015stability}, \cite{li2015f}. 

The compactness theorem relies on two main ingredients. The first ingredient is a curvature estimate. In \cite{choi1985space} Choi and Schoen developed a blow up technique to get a point-wise curvature bound for minimal surfaces with small total curvature. Here we will use the same technique, but develop the necessary tools for $\lambda$-surfaces. 

The second ingredient is the smoothness of the convergence. We will follow the idea of Colding and Minicozzi in \cite{colding2012smooth}. Their idea is if the sequence of self-shrinkers do not convergence smoothly to a limit self-shrinker, then we can construct a positive Jacobi field of the second variational operator of self-shrinkers, which is impossible by a result in \cite{colding2012generic}. In this paper, we observe that although $\lambda$-surfaces do not satisfy the same equations for different $\lambda$'s, the major linear terms of the difference between $\lambda$-surfaces are the same. Thus we can still get some uniform estimate just like the self-shrinker case.

If the limit is a self-shrinker, i.e. $\lambda=0$, then the multiplicity $2$ convergence is possible. Suppose $\Sigma_i$ are $\lambda_i$-surfaces converge to $\Sigma$ which is a self-shrinker, then it is possible that $\Sigma_i$ locally is double sheeting graphs over $\Sigma$, with different orientation. Then the linearized operator is not the same as the second variational operator over $\Sigma$. However, if the number of sheets is larger than $2$, then we can find two sheets have the same orientation, and considering the sheets with this orientation we can get a positive Jacobi field as well. Thus the convergence if at most of multiplicity $2$.

Note for general $\lambda\neq 0$, $\lambda$-surfaces do not have a maximum principle, so touching may happen in the limit surface. There are two limiting models of touching points. One is called {\bf kissing}, and we can image two disks are far from each other at first and then try to kiss each other. Finally they really kiss each other to form a touching; Another one is called {\bf neck pinching}, and we can image a very thin neck is pinching and finally collapse to form a touching. It turns out that the smooth convergence can actually passing through those kissing touching points, but can never be smooth at neck pinching points, because the topology changes during this limiting process.

We will discuss more details in a paper in preparation.
\subsection{Organization of the Paper}

In section 2 we will discuss the properties of the second variational operator $L$ for $\lambda$-surfaces. We will prove $L$ is the linearized operator of $\lambda$-surfaces, hence the solution to this operator carries the information of the nearby $\lambda$-surfaces.

In section 3 we prove the main compactness theorem. We first follow \cite{choi1985space} to get curvature estimate, so that we can find convergence subsequence; then we follow the idea of \cite{white1987curvature} to prove the limit is smoothly immersed; finally we follow the idea of \cite{colding2012smooth} to show the convergence is smooth.

In section 4 we will prove the rigidity theorem. The main idea is to use the compactness theorem to analyze the linearized operators.

\subsection{Acknowledgement}
The author want to thank Professor Bill Minicozzi for his helpful advises and heuristic conversations. We also want to thank Jonathan Zhu for pointing our the possibility of multiplicity $2$ convergence.

\section{Linearized Operator}
Throughout this section, we will assume the $\lambda$-surface $\Sigma\subset\mb R^3$ is smooth immersed with polynomial area growth.
\subsection{Property of Linearized Operator}
We define the operator
\begin{equation}
L:=\Delta_\Sigma-\frac{1}{2}\langle \nabla_{\Sigma}\cdot,x\rangle+\vert A\vert^2+\frac{1}{2}
\end{equation}
where $A$ is the second fundamental form of the surface.

If $\Sigma$ is a self-shrinker, $L$ is the second variational operator with respect to the Gaussian weight of the surface. See \cite{colding2012generic} for more details. One important property of $L$ is that for any vector $v\in\mb R^3$, $\langle v,\mathbf{n}\rangle$ is an eigenfunction of $L$ with eigenvalue $1/2$. This is obtained in \cite{mcgonagle2015hyperplane}, \cite{guang2014gap}.

\begin{lemma}\label{qianglem}
If $\Sigma\in\mb R^3$ is a $\lambda$-surface, then for any constant vector $v\in\mb R^3$, we have
\[L\langle v,\mathbf{n}\rangle=\frac{1}{2}\langle v,\mathbf{n}\rangle.\]
\end{lemma}

\begin{proof}
Fix a constant vector $v$, and set $f=\langle v,\mathbf{n}\rangle$. Let $e_1,e_2$ be an orthonormal frame for $\Sigma$ at $p$, then
\[\nabla_{e_i}f=\langle v,\nabla_{e_i}\mathbf{n}\rangle=-a_{ij}\langle v,e_j\rangle.\]
here $a_{ij}$ is the second fundamental form under the local frame. Differentiating again and using the Codazzi equation gives at $p$ that
\[\nabla_{e_k}\nabla_{e_i}f=-a_{ik,j}\langle v,e_j\rangle-a_{ij}\langle v,a_{jk\mathbf{n}}\rangle.\]

Taking the trace gives
\begin{equation}
\Delta_\Sigma f=\langle v,\nabla H\rangle-\vert A\vert^2f.
\end{equation}

By the $\lambda$-surface equation $H=\frac{\langle x,\mathbf{n}\rangle}{2}+\lambda$, we get
\begin{equation}
\nabla_{e_i}H=\langle e_i,\mathbf{n}\rangle+\langle x,\nabla_{e_i}\mathbf{n}\rangle=-a_{ij}\langle x,e_j\rangle.
\end{equation}
Combining these two identities gives the lemma.

\end{proof}

With the help of this lemma, we can follow \cite{colding2012generic} and \cite{colding2012smooth} to prove the following non-existence theorem

\begin{theorem}\label{nopositivesol}
For any $\lambda$-surface $\Sigma$, there is no positive function $u$ on $\Sigma$ with $Lu=0$.
\end{theorem}

\begin{proof}
Set $w=\log u$, so that
\[\Delta w=\frac{\delta u}{u}-\vert\nabla w\vert^2=-\vert A\vert^2+\frac{1}{2}\langle x,\nabla w\rangle-\frac{1}{2}-\vert\nabla w\vert^2.\]
Then given $\phi$ with compact support on $\Sigma$, applying Stokes theorem to $\dv(\phi^2e^{\frac{-\vert x\vert^2}{4}}\nabla w)$ gives
\begin{equation}
\begin{split}
0&=\int_{\Sigma}(2\phi\langle\nabla\phi,\nabla w\rangle+(-\vert A\vert^2-\frac{1}{2}-\vert\nabla w\vert^2)\phi^2)e^{\frac{-\vert x\vert^2}{4}}\\
&\leq \int_{\Sigma}(\phi^2\vert\nabla w\vert^2+\vert\nabla\phi\vert^2-\vert A\vert^2-\frac{1}{2}\vert A\vert^2-\vert\nabla w\vert^2\phi^2)e^{\frac{-\vert x\vert^2}{4}}\\
&\leq\int_{\Sigma}(\vert\nabla\phi\vert^2-\vert A\vert^2\phi^2-\frac{1}{2}\phi^2)e^{\frac{-\vert x\vert^2}{4}}=-\int_{\Sigma}(\phi L\phi)e^{\frac{-\vert x\vert^2}{4}}.
\end{split}
\end{equation}
On the other hand, fix a point $p$ in $\Sigma$ and define a function $v(x)=\langle \mathbf n(p),\mathbf n(x)\rangle$. Then we know $v(p)=1$, $\vert v(p)\vert\leq1$, and by Lemma \ref{qianglem} $Lv=\frac{1}{2}v$. So we have
\[L(\eta v)=\frac{1}{2}\eta v+v(\Delta\eta-\frac{1}{2}\langle x,\nabla \eta\rangle)+2\langle\nabla\eta,\nabla v\rangle.\]
Then integrate $\eta v L(\eta v)$ we get that
\begin{equation}
\begin{split}
-\int_\Sigma\eta vL(\eta v)e^{\frac{-\vert x\vert^2}{4}}&=-\int_{\Sigma}(\frac{1}{2}\eta^2 v^2+\eta v^2(\Delta \eta-\frac{1}{2}\langle x,\nabla\eta\rangle)+\frac{1}{2}\langle\nabla\eta^2,\nabla v^2\rangle)e^{\frac{-\vert x\vert^2}{4}}\\
&=-\int_{\Sigma}(\frac{1}{2}\eta^2v^2-v^2\vert\nabla\eta\vert^2)e^{\frac{-\vert x\vert^2}{4}}.
\end{split}
\end{equation}
where the second inequality uses Stokes' theorem. Now we choose $\eta$ be identically $1$ on $B_R$ and cut-off linearly to $0$ on $B_{R+1}\setminus B_{R}$, we get

\begin{equation}
-\int_{\Sigma}\eta vL(\eta v)e^{\frac{-\vert x\vert^2}{4}}\leq\int_{\Sigma\setminus B_{R}}v^2e^{\frac{-\vert x\vert^2}{2}}-\frac{1}{2}\int_{\Sigma\cap B_R}v^2e^{\frac{-\vert x\vert^2}{4}}.
\end{equation}
Since $\Sigma$ has polynomial area growth and $\vert v\vert\leq1$, we have
\[\lim_{R\to\infty}\int_{\Sigma\setminus B_R}v^2 e^{\frac{-\vert x\vert^2}{4}}= 0.\]
So for $R$ sufficiently large, we have
\[\int_{\Sigma}\eta v L(\eta v)e^{\frac{-\vert x\vert^2}{4}}<0.\]

Then we get a contradiction if we set $u=\eta v$.
\end{proof}

\subsection{Difference of Two $\lambda$-Surfaces}
In previous section we defined the linearized operator to be the second variational operator of Gaussian weight of $\lambda$-surfaces. The reason it is called linearized operator is that for two $\lambda$-surfaces close enough, then their difference satisfies an second order elliptic equation, and the main linear term is just this linearized operator. This is a well-known result in minimal surfaces case, see \cite{simon1987strict}, \cite{kapouleas1990complete}.

\begin{theorem}\label{difference}
Suppose $\Sigma_1$ is a $\lambda_1$-surface and $\Sigma_2$ is a $\lambda_2$-surface. If $\Sigma_2$ can be written as a graph on $\Sigma_1$, namely suppose $\Sigma_1$ is a map $X:\Sigma_1\to\mb R^3$, then $\Sigma_2$ is $X+\varphi\mathbf{n}$ for some $C^2$ function $\varphi$. If on $\Sigma_1$, $\vert\varphi A\vert<1$, then $\varphi$ satisfies the equation
\begin{equation}
L\varphi+(\lambda_2-\lambda_1)=\dv(a\cdot\nabla\varphi)+b\cdot\nabla\varphi+c\varphi.
\end{equation}

Moreover if $\Sigma_i$ is a sequence of $\lambda_i$-surfaces converge to a $\lambda$-surface $\Sigma$ in the $C^2$ sense, i.e. $\Sigma_i=\Sigma+\varphi_i\mathbf{n}$ and $\varphi_i$ are $C^2$ converges to $0$. Suppose 
\begin{equation}
L\varphi_i+(\lambda_i-\lambda)=\dv(a_i\cdot\nabla\varphi_i)+b_i\cdot\nabla\varphi_i+c_i\varphi_i.
\end{equation}
Then $a_i,b_i,c_i$ all uniformly converges to $0$, and the right hand side of the equation uniformly converges to $0$.
\end{theorem}

\begin{proof}
In this proof $H_i$ will be the mean curvature on $\Sigma_i$. We refer to Kapouleas' paper \cite{kapouleas1990complete}, Appendix 3. Note in \cite{kapouleas1990complete}, his $H$ is our $H$ with a factor $-\frac{1}{2}$. 

We follow Kapouleas to use $\Phi$ to denote a term which can be either $\varphi A$ or $\nabla\varphi$, and use $\star$ to denote contraction with respect to the metric on $\Sigma_1$. $G_i$ stands for linear combinations with universal coefficients of terms which are contractions with respect to the metric on $\Sigma_1$ of at least two $\Phi$'s, and $\tilde G_i$ stands for linear combinations with universal coefficients of terms which are contraction of a number of (possibly none) $\Phi$'s with one of the following:
\begin{itemize}
\item $A\star\Phi\star\Phi$
\item $\varphi\nabla A\star\Phi$
\item $\varphi A\star\nabla^2\varphi$
\item $\nabla^2\varphi\star\Phi\star\Phi$
\end{itemize}
So all the terms are linear in  at most the second derivative of $\varphi$.

By Lemma C.2 in \cite{kapouleas1990complete}, if $\vert\varphi A\vert<1$ we have
\begin{equation}
H_1-H_2-(\Delta\varphi+\vert A\vert^2\varphi)=\frac{\tilde G_1}{\sqrt{1+G_1}}+\frac{\tilde G_2}{1+G_1+\sqrt{1+G_1}}.
\end{equation}
Plugging in the equation of $\lambda$-surfaces gives
\begin{equation}
-(\Delta\varphi+\vert A\vert^2\varphi)+\frac{1}{2}\langle x,\mathbf{n}\rangle-\frac{1}{2}\langle x+\varphi\mathbf n,\mathbf{n}\rangle+\lambda_1-\lambda_2=\frac{\tilde G_1}{\sqrt{1+G_1}}+\frac{\tilde G_2}{1+G_1+\sqrt{1+G_1}}.
\end{equation}
Then we get
\begin{equation}
L\varphi+\lambda_2-\lambda_1=-\frac{\tilde G_1}{\sqrt{1+G_1}}-\frac{\tilde G_2}{1+G_1+\sqrt{1+G_1}}=\dv(a\nabla\varphi)+b\cdot\nabla\varphi+c\varphi.
\end{equation}
where we write the coefficients to $a,b,c$. When $\varphi\to 0$ in $C^2$, since $G_i,\tilde G_i$ has universal coefficients, and $\Phi$ terms will uniformly converge to $0$, we get $a,b,c$ uniformly converge to $0$.
\end{proof}
\begin{remark}\label{remarklinearlizedequation}
We can get more information from the calculation
\begin{equation}
L\varphi+\lambda_2-\lambda_1=-\frac{\tilde G_1}{\sqrt{1+G_1}}-\frac{\tilde G_2}{1+G_1+\sqrt{1+G_1}}.
\end{equation}
Let the right hand side of the equation to be $P(x,\varphi,\nabla \varphi,\nabla^2\varphi)$. Then from the computation in K paper,
\begin{equation}
\Vert P(x,\varphi,\nabla \varphi,\nabla^2\varphi)\Vert_{C^1}\leq C\Vert \varphi\Vert_{C^1}.
\end{equation}
More precisely, in the definition of $G$ and $\tilde G$, each $\varphi,\nabla\varphi,\nabla^2\varphi$ term contracts with at least one other $\varphi$ or $\nabla\varphi$ term. So the coefficients of linearization of $P$ must also consist at least one other $\varphi$ or $\nabla\varphi$ term. So it is bounded by $C^1$ norm of $\varphi$.

This bound would play a role in later section.
\end{remark}

\section{Compactness}\label{Section:Compactness}
\subsection{Choi-Schoen Type Curvature Estimate}
In this section we are going to prove the main compactness theorem. In order to prove the compactness theorem, we need to get the curvature bound of $\lambda$-surfaces. First we show there is a uniform integral bound for $\lambda$-surface if $\lambda$ is bounded. We need the local Gauss-Bonnet estimate by Ilmanen \cite{ilmanen1995singularities}.

\begin{theorem}[Theorem 3 in \cite{ilmanen1995singularities}]\label{GaussBonnet}
Let $R>r>0$ and let $M$ be a $2$-manifold properly immersed in $B_R$. Then for any $\eps>0$
\begin{equation}
(1-\eps)\int_{M\cap B_r}\vert A\vert^2d\mu\leq\int_{M\cap B_R}H^2d\mu+8\pi g(M\cap B_R)+\frac{24\pi D'R^2}{\eps(R-r)^2}.
\end{equation}
Here 
\[D'=\sup_{s\in[r,R]}\frac{\area(M\cap B_s)}{\pi s^2}.\]
\end{theorem}

Note for a $\lambda$-surface $\Sigma$, its mean curvature satisfies $\vert H\vert=\vert\frac{1}{2}\langle x,N\rangle+\lambda\vert\leq \vert x\vert+\vert\lambda\vert$. Then we choose $\eps=1/2$ in the theorem we get the following uniformly integral curvature bound:

\begin{corollary}
Suppose $\Sigma$ is a $\lambda$-surface with $\vert\lambda\vert\leq \Lambda$, genus no more than $g$ and $\area(B_s(x_0)\cap\Sigma)\leq Vs^2$ for $r<s<2r$, then we get uniformly integral curvature bound
\begin{equation}
\int_{M\cap B_r}\vert A\vert d\mu\leq C(\Lambda,g,V,x_0).
\end{equation}
\end{corollary}
\begin{proof}
Just choose $\eps=1/2$ and $R=2r$ in the local Gauss-Bonnet estimate.
\end{proof}

Then we will follow the idea of Choi-Schoen in \cite{choi1985space} to prove a pointwise curvature bound from a small integral curvature bound. Choi-Schoen did this estimate for minimal surfaces in three manifold, and later Colding-Minicozzi did this estimate for self-shrinkers in $\mb R^3$. We need to generalize this type of estimate to $\lambda$-surfaces.

\begin{theorem}\label{CS}
Let $\Sigma$ be a $\lambda$-surface in $\mb R^3$ without boundary, $\vert\lambda\vert\leq\vert\Lambda\vert$. Let $x_0$ be a point in $\Sigma$. Then there exists $\eps_0>0$ such that if
\[\int_{\Sigma\cap B_r(x_0)}\vert A\vert^2\leq\eps_0\] and $r\leq\eps_0$, then the following inequality holds
\begin{equation}
\max_{0\leq\sigma\leq r}\sigma^2\sup_{B_{r-\sigma}(x_0)}\vert A\vert^2\leq C.
\end{equation}
where $C$ is a constant depending on $\Lambda$ and $\vert x_0\vert$. 
\end{theorem}

Before the proof of Theorem \ref{CS}, we need some lemmas for $\lambda$-surfaces. The first two lemmas are computations of the drift Laplacian of various quantities. See \cite{colding2012generic}. Recall that from \cite{colding2012generic} the drift Laplacian operator $\mc L$ is defined by
\[\mc L=\Delta-\frac{1}{2}\langle x,\nabla\cdot\rangle=e^{\frac{\vert x\vert^2}{4}}\dv(e^{\frac{-\vert x\vert^2}{4}}\nabla\cdot ).\]

\begin{lemma}\label{simons}
Suppose $\Sigma$ is a $\lambda$-surface with $\vert\lambda\vert\leq\Lambda$, then
\begin{equation}
\mc L\vert A\vert^2=2(\frac{1}{2}-\vert A\vert^2)\vert A\vert^2-2\lambda\langle A^2,A\rangle+2\vert\nabla A\vert^2\geq -C(\Lambda)(\vert A\vert^2+\vert A\vert^4).
\end{equation}
\end{lemma}
\begin{proof}
We refer the proof to Lemma 2.1 in \cite{guang2014gap}.
\end{proof}

\begin{lemma}\label{lemmaLx}
For $x_0$ a fixed point in $\mb R^3$,
\begin{equation}
\mc L\vert x-x_0\vert^2=-\langle x,x-x_0\rangle-2\lambda\langle\mathbf n,x-x_0\rangle+4.
\end{equation}
\end{lemma}

\begin{proof}\label{driftx}
We will show the result for $n$ dimensional $\lambda$-hypersurface in $\mb R^{n+1}$. We follow the proof in \cite{colding2012generic}. Note $\Delta x=-H\mathbf n$, so we have
\begin{equation}
\begin{split}
\Delta\langle x-x_0,x-x_0\rangle&=2\langle\Delta (x-x_0),x-x_0\rangle+2\vert\nabla(x-x_0)\vert^2\\
&=-2\langle H\mathbf{n},x-x_0\rangle+2n\\
&=-\langle x,\mathbf n \rangle\langle \mathbf n,x-x_0\rangle-2\lambda\langle\mathbf{n},x-x_0\rangle+2n\\
&=-\langle x,(x-x_0)^\bot\rangle-2\lambda\langle\mathbf{n},x-x_0\rangle+2n.
\end{split}
\end{equation}
This identity combined with the fact that
\[\frac{1}{2}\langle x,\nabla\vert x-x_0\vert^2\rangle=\langle x,(x-x_0)^\top\rangle\]
gives the identity.
\end{proof}

\begin{remark}\label{remarklemmaLx}
Let the point $x_{\text{min}}$ achieve the minimum of $\vert x\vert$ on $\Sigma$, then set $x_0=0$ we get
\[\vert x_{\text{min}}\vert\leq\sqrt{\lambda^2+4}-\lambda.\]
If $\Sigma$ is compact, consider the point $x_{\text{max}}$ achieve the maximum of $\vert x\vert$ on $\Sigma$, then set $x_0=0$ we get
\[\vert x_{\text{max}}\vert\geq\sqrt{\lambda^2+4}-\lambda.\]

In conclusion, any $\lambda$-surface must intersect with the sphere $S^2_{r=\sqrt{\lambda^2+4}-\lambda}$.

\end{remark}

The third lemma is a monotonicity formula for $\lambda$-surfaces.

\begin{lemma}\label{monotonicity}
Suppose $\Sigma^{n}\subset\mb R^{n+1}$ is a $\lambda$-hypersurface. Let $x_0\in\Sigma$ and $B_s=B_s(x_0)$ be the ball centered at $x_0$. Suppose $f$ is a nonnegative function satisfying $\mc L f\geq -ct^{-2}f$, then 
\begin{equation}
f(x_0)\leq C(\Lambda,c,t,x_0)\int_{B_t(x_0)\cap\Sigma}f. 
\end{equation}
\end{lemma}

\begin{proof}
The proof is based on section 1.3 of \cite{colding2011course}. By Lemma \ref{driftx} we have
\begin{equation}
2n\int_{B_s\cap\Sigma}fe^{\frac{-\vert x\vert^2}{4}}=\int_{B_s\cap\Sigma}f\mc L(\vert x-x_0\vert^2)e^{\frac{-\vert x\vert^2}{4}}+\int_{B_s\cap \Sigma}f\langle x,x-x_0\rangle e^{\frac{-\vert x\vert^2}{4}}+\int_{B_s\cap\Sigma}2\lambda f\langle\mathbf{n},x-x_0\rangle e^{\frac{-\vert x\vert^2}{4}}.
\end{equation}

Applying integration by part to the first term on right hand side gives

\begin{equation}
\begin{split}
\int_{B_s\cap\Sigma}f\mc L(\vert x-x_0\vert^2)e^{\frac{-\vert x\vert^2}{4}}&=\int_{B_s\cap\Sigma}f\mc L(\vert x-x_0\vert^2-s^2)e^{\frac{-\vert x\vert^2}{4}}
=\int_{B_s\cap\Sigma}f\dv(e^{\frac{-\vert x\vert^2}{4}}(\vert x-x_0\vert^2-s^2))\\
&=\int_{B_s\cap\Sigma}\mc L f(\vert x-x_0\vert^2-s^2)e^{\frac{-\vert x\vert^2}{4}}+\int_{\partial B_s\cap\Sigma}f \nu\cdot e^{\frac{-\vert x\vert^2}{4}}\nabla(\vert x-x_0\vert^2-s^2))\\
&=\int_{B_s\cap\Sigma}\mc L f(\vert x-x_0\vert^2-s^2)e^{\frac{-\vert x\vert^2}{4}}+2\int_{\partial B_s\cap\Sigma}f\vert (x-x_0)^\top\vert e^{\frac{-\vert x\vert^2}{4}}.
\end{split}
\end{equation}
where $\nu$ is the boundary outer normal of $B_s\cap\Sigma$.

Using the above identity and the coarea formula gives
\begin{equation}
\begin{split}
\frac{d}{ds}(s^{-n}\int_{B_s\cap\Sigma}fe^{\frac{-\vert x\vert^2}{4}})
&=-ns^{-n-1}\int_{B_s\cap\Sigma}fe^{\frac{-\vert x\vert^2}{4}}+s^{-n}\int_{\partial B_s\cap\Sigma}f\frac{\vert x-x_0\vert}{\vert (x-x_0)^\top\vert}e^{\frac{-\vert x\vert^2}{4}}\\
&=-s^{-n-1}\frac{1}{2}\int_{B_s\cap \Sigma}f\langle x,x-x_0\rangle e^{\frac{-\vert x\vert^2}{4}}-s^{-n-1}\int_{B_s\cap\Sigma}\lambda f\langle\mathbf{n},x-x_0\rangle e^{\frac{-\vert x\vert^2}{4}}\\
&-s^{-n-1}\frac{1}{2}\int_{B_s\cap\Sigma}\mc L f(\vert x-x_0\vert^2-s^2)e^{\frac{-\vert x\vert^2}{4}}+s^{-n-1}\int_{\partial B_s\cap\Sigma}f\frac{\vert(x-x_0)^\bot\vert^2}{\vert (x-x_0)^\top\vert^2}e^{\frac{-\vert x\vert^2}{4}}.
\end{split}
\end{equation}
Now let 
\[g(s)=s^{-n}\int_{B_s\cap\Sigma}fe^{\frac{-\vert x\vert^2}{4}}.\]
The previous computation implies that 
\begin{equation}
g'(s)\geq-Cg(s)-Cg(s)-cst^{-2}g(s),
\end{equation}
where $C$ is a positive constant depending on $x_0$, $\Lambda$. Here note on $B_s\cap\Sigma$, $\vert x-x_0\vert\leq s$. Then we have
\[\frac{g'(s)}{g(s)}\geq-C-\frac{cs}{t^2}\geq-C-\frac{c}{t}.\]
So $e^{(C+\frac{c}{t})s}g(s)$ is monotone nondecreasing. Thus we get
\begin{equation}\label{monotonicity2}
f(x_0)e^{\frac{-\vert x_0\vert^2}{4}}\leq e^{Ct+c}\omega^{-1}t^{-n}\int_{B_t(x_0)\cap\Sigma}fe^{\frac{-\vert x\vert^2}{4}}.
\end{equation}
where $\omega$ is the volume of unit $n$-ball in $\mb R^{n+1}$. We simplify the terms to get
\begin{equation}
f(x_0)\leq C(\Lambda,x_0,c)e^{C(\Lambda,x_0)t}t^{-n}\int_{B_t(x_0)\cap\Sigma}f= C(\Lambda,c,t,x_0)\int_{B_t(x_0)\cap\Sigma}f.
\end{equation}
\end{proof}

\begin{remark}\label{remarkareabound}
If we pick $f=1$ in the above estimate, we will see that
\[\frac{\area(\Sigma\cap B_{s})(x_0)}{s^2}\leq C(x_0,\Lambda)e^{C(x_0,\Lambda)t}\frac{\area(\Sigma\cap B_{t}(x_0))}{t^2}\]
if $s<t$ and $t>1$. As a result, if the sequence $\{\Sigma_i\}$ in Theorem \ref{mainthm} are compact surfaces, we only need uniform diameter bound and area bound to get compactness.
\end{remark}

Now we have gathered all the ingredients to prove the curvature estimate.

\begin{proof}[Proof of Theorem \ref{CS}]
Let $\sigma_0\in(0,r]$ be chosen so that 
\[\sigma_2^2\sup_{B_{r-\sigma_0}(x_0)}\vert A\vert^2=\max_{0\leq\sigma\leq r}\sigma^2\sup_{B_{r-\sigma}(x_0)}\vert A\vert^2.\]

Let $y\in B_{r-\sigma_0}(x_0)$ be the point such that
\[\vert A\vert^2(y)=\sup_{B_{r-\sigma_0}(x_0)}\vert A\vert^2.\]
Hence
\[\sup_{B_{\sigma_0/2}(y)}\vert A\vert^2\leq 4\vert A\vert^2(y).\]

If $\sigma_0^2\vert A\vert^2(y)\leq 4$, then the inequality is true. So we may assume that
\begin{equation}\label{CS2}
\alpha:=\vert A\vert^2(y)\geq4\sigma_0^{-2}.
\end{equation}

Then by Lemma \ref{simons}, we have $\mc L\vert A\vert^2\geq -(C(\Lambda)+\alpha)\vert A\vert^2$, then applying Lemma \ref{monotonicity} with $f=\vert A\vert^2$, $t=\sigma_0/2$, $x_0=y$ and $c=\sigma_0^2/4(C(\lambda)+\alpha)$ we get
\begin{equation}
\vert A\vert^2(y)\leq C(\sigma_0,y,\Lambda,\alpha)\int_{B_{\sigma_0/2}(y)\cap\Sigma}\vert A\vert^2\leq C(\sigma_0,y,\Lambda,\alpha)\epsilon_0.
\end{equation}

Moreover, keep track of the $c$ term in (\ref{monotonicity2}) we can find a uniform bound $C(\Lambda,y)$ for any $\sigma_0\leq 1$ and $\alpha\geq 4\sigma_0^2$. i.e.

\begin{equation}
\vert A\vert^2(y)\leq C(y,\Lambda)\epsilon_0.
\end{equation}

So if $\eps_0$ is small enough, the inequality contradicts \ref{CS2}. Then we finish the proof.

\end{proof}

\begin{remark}
Choi-Schoen used a blow-up argument in the proof for convenience. In their case, minimal surface is still minimal after rescaling, but in our case, $\lambda$-surface is no longer a $\lambda$-surface after rescaling. Hence we just carefully keep track of every term in the monotonicity formula and get the same kind of contradiction.
\end{remark}

With this curvature bound, we have the following compactness theorem with singular points:

\begin{theorem}\label{convergence1}
Suppose we have a sequence of $\lambda_i$-surfaces $\Sigma_i$ with genus at most $g$ and $\partial\Sigma_i=\emptyset$ satisfying the following conditions:
\begin{enumerate}
\item $\area(B_R(x_0)\cap\Sigma)\leq C(x_0)R^2$ for all $x_0\in\mb R^3$ and all $R>0$
\item $\vert\lambda_i\vert\leq\Lambda$
\end{enumerate}
Then there is a subsequence (still denoted by $\Sigma_i$), a smooth embedded complete non-trivial $\lambda_\infty$-surface $\Sigma$ without boundary but locally finite collection of points $\mc S\subset\Sigma$ removed, so that $\Sigma_i$ converges smoothly (possibly with multiplicities) to $\Sigma$ off of $\mc S$.
\end{theorem}

\begin{proof}
We follow \cite{choi1985space} and \cite{colding2012smooth}. For any ball $B_R(x)\subset\mb R^3$, by Theorem \ref{GaussBonnet}, the total curvature of any $\Sigma_i$ over $B_R(x)\cap\Sigma_i$ is uniformly bounded by some constant $C$. For each positive integer $m$, take a finite covering $\{B_{r_m}(y_j)\}$ of $B_R(x)$ such that each point of $B_R(x)$ is covered at most $h$ times by balls in this covering, and $\{B_{r_m/2}(y_j)\}$ is still a covering of $B_R(x)$. Here we set $r_m=2^{-m}\eps_0$ and $h$ only depends on $B_R(x)$. Then we have
\[\sum_{j}\int_{\Sigma_i\cap B_{r_m}(y_j)}\vert A\vert^2\leq hC\]
Therefore for each $i$ at most $hC/\eps_0$ number of balls on which
\[\int_{\Sigma_i\cap B_{r_m}(y_j)}\vert A\vert^2\geq\eps_0\]
By passing to a subsequence of $\Sigma_i$ we can always assume that all the $\Sigma_i$ has the same balls with total curvature $\geq\eps_0$. Call the center of these balls to be $\{x_{1,m},\cdots,x_{l,m}\}$, where $l$ is an integer at most $hC/\eps_0$. Then on the balls other than $B_{x_{k,m}(r_m)}$, by Theorem \ref{CS} we have uniformly point-wise curvature bound for $\Sigma_i$. Passing to subsequence we may assume that $\Sigma_i$ converges smoothly on a half of those balls to $\Sigma$.

Then we can continue this process as $m$ increase. Finally by a diagonal argument we can get a subsequence $\{\Sigma_i\}$, converges smoothly everywhere to $\Sigma$ besides those points $x_1,\cdots,x_l$ which is the limit of those $\{x_{1,m}\},\cdots,\{x_{l,m}\}$. Moreover, since there is no maximum principle for $\lambda$-surfaces, the limit is only immersed. However if we consider the compactness for each connected components in any fixed ball, we can see the limit is self-touching.

Finally, by Lemma \ref{lemmaLx} and the remark, we know that any $\lambda$-surface must intersect with the union of spheres $S^2_{r=\sqrt{\lambda_i^2-4}-\lambda_i}(0)$. So the limit must be non-trivial.
\end{proof}

\subsection{Removable Singularities}

Next we want to show the limit $\Sigma$ in Theorem \ref{convergence1} is actually smooth. \cite{colding2012smooth}. 

First, self-shrinkers are actually minimal surfaces under some conformal change of the ambient metric, and the removable of singularities is already known in Choi-Schoen. In our case, $\lambda$-surfaces are not minimal under certain conformal change of metric. Hence we need another approach to prove the removable of singularities. Here we will follow White \cite{white1987curvature}. 

Second, self-shrinkers satisfies the maximal principle, but here for $\lambda$-surfaces we only have maximal principle in a more constraint sense. As a result, touching can not appear under the convergence of self-shrinkers, but may appear under the convergence of $\lambda$-surfaces. So we need extra assumption of the convergence to avoid touching under convergence.

So we need more delicate curvature estimate for our limit. In later discussion, we sometimes need to translate and rescale $\lambda$ surfaces, so we list some observation here. Suppose $\Sigma$ is a $\lambda$-surface, then after translation by $-z$ and rescaling by $\alpha$ ($x\to\alpha x$), the new surface $\tilde\Sigma$ satisfies the following equation 

\begin{equation}\label{rescalesurfaceequation}
H=\frac{1}{2\alpha^2}\langle x,\mathbf{n}\rangle+\frac{\lambda}{\alpha}.
\end{equation}

In particular, the right hand side of the equation goes to $0$ when $\alpha\to+\infty$.

Now we show the main curvature estimate in this section. The idea is based on \cite{white1987curvature}.

\begin{lemma}\label{curvaturenearsingularity}
Suppose $\Sigma$ is a properly self-touching $\lambda$-surface in $B_r(x_0)\setminus \{x_0\}$ with $\vert\lambda\vert\leq\Lambda$ and $r\leq R$ then there exists $\eps=\eps(\Lambda,R,x_0)>0$ such that if $\int_{\Sigma}\vert A\vert^2\leq\eps$, then there is $C$ such that
\begin{equation}
\vert A(x)\vert\vert x-x_0\vert\leq C.
\end{equation}
\end{lemma}

\begin{proof}
We show by contradiction. If the criterion is not true, then we can find a sequence of points $x_n\in((B_{R}(x_0)\setminus B_{1/n}(x_0))\cap\Sigma)$ such that
\[\vert A(x_n)\vert^2(\vert x_n-x_0\vert-\frac{1}{n})\to+\infty.\]
 Otherwise we will have uniform bound for $\vert A(x)\vert^2(\vert x-x_0\vert-\frac{1}{n})$ for a sequence of $n\to\infty$, then passing to limit we will have a uniform bound for $\vert A(x)\vert^2\vert x-x_0\vert$.

Then we can choose $z_n\in ((B_{R}(x_0)\setminus B_{1/n}(x_0))\cap\Sigma)$ such that $\vert A(z_n)\vert^2(\vert z_n-x_0\vert-\frac{1}{n})$ achieve maximum. Note that $\vert A(x)\vert^2(\vert x-x_0\vert-\frac{1}{n})$ achieve $0$ on $\partial B_{1/n}(x_0)\cap\Sigma$, so $d_n:=\vert z_n-x_0\vert-\frac{1}{n}>0$.

Now we translate $\{x\in\Sigma:\vert x-z_n\vert\leq d_n/2\}$ in $\mb R^3$ such that $z_n$ becomes $0$ the origin, and then rescale this part of the surface with scale $\vert A(z_n)\vert$. We denote this new surface by $\tilde\Sigma_n$, and use tilde to denote the quantities on this new surface.

Note $\tilde\Sigma_n$ satisfy the following properties. First, $\vert \tilde A(0)\vert=1$; Second, by 
\[\vert A(z_n)\vert^2d_n\to+\infty\]
we know that for any fixed $R>0$, $\tilde\Sigma_n\cap \partial B_{R}(0)\not=\emptyset$ if $n$ large enough, and $\partial\tilde\Sigma_n\cap B_R(0)=\emptyset$ if $n$ large enough; Finally, for any $x'=\vert A(z)\vert x\in\tilde\Sigma_n$, we have

\[\vert A(x)\vert(\vert x-x_0\vert-\frac{1}{n})\leq \vert A(z)\vert d_n.\]

Since $\vert x-z\vert\leq d_n/2$, we have $\vert x-x_0\vert-\frac{1}{n}\geq d_n/2$, thus $\vert A(x)\vert\leq 2\vert A(z)\vert$, thus $\vert \tilde A(x')\vert\leq 2$.
 
By the uniform curvature bound for $\tilde\Sigma_n$, for each $R>0$, there exists a subsequence (still denoted by $\tilde\Sigma_n$) converging smoothly on $B_R(0)$ to a complete surface $\tilde\Sigma$. Checking the equation of translation and rescaling, we see that the limit $\tilde\Sigma$ must be a minimal surface, i.e. $\tilde H=0$.

Since the rescaling would not change the integral of the squared curvature, we have

\[\int_{B_R(0)\cap\tilde\Sigma}\vert A\vert^2\leq\eps.\]

Thus $\tilde\Sigma$ has to be the plane. Which contradicts to the condition that $\vert \tilde A(0)\vert=1$.

\end{proof}

\begin{theorem}\label{limit}
The limit surface in Theorem \ref{convergence1} is smoothly immersed. Moreover, for $y\in \mc S$ be a non-embedded point, in a small neighbourhood of $y$, $\Sigma$ is a union of two disks which are touching at $y$.
\end{theorem}

\begin{proof}
By appropriate translation of $\Sigma$ suppose $x_0\in\mc S$ and $r>0$ such that $0$ is the only element in $\mc S$ in $B_r(0)$. Since $\int_{\Sigma_i\cap B_r(0)}\vert A\vert^2$ has uniform bound by generalized Gauss-Bonnet theorem \ref{GaussBonnet}, we may assume $r$ small enough such that $\int_{B_r(0)\cap\Sigma}\vert A\vert^2\leq\eps$.

By Lemma \ref{curvaturenearsingularity}, there is a constant $C$ such that
\[\vert A(x)\vert\vert x\vert\leq C.\]
for any $x\in B_r(0)\cap\Sigma$. Now we choose a sequence $r_i\to 0$ and rescale $\Sigma$ by $1/r_i$ and denote it by $\tilde\Sigma_i$. Note the curvature bound
\[\vert A(x)\vert\vert x\vert\leq C\]
is invariant under rescaling, so this uniform curvature bound indicate that $\tilde\Sigma_i$ smoothly converges to a complete surface $\tilde\Sigma$ in $\mb R^3\setminus\{0\}$. See \cite{white1987curvature}.

Now for $K$ be any compact subset of $\mb R^3\setminus\{0\}$, 
\[\int_{\tilde\Sigma_i\cap K}\vert A\vert^2=\int_{\Sigma_i\cap r_iK}\vert A\vert^2\to0\text{ as $r_i\to0$}.\]
This implies $\tilde\Sigma$ is a union of planes. Thus $\Sigma\cap B_r(0)$ is actually a union of disks and punctured disks (see\cite{white1987curvature}).

Now let $\Sigma$ just denote one of its connected components which is a punctured disk. Since $\tilde\Sigma_i$ converges to to the plane in $\mb R^3\setminus\{0\}$, we can assume for some $i$, $\tilde\Sigma_i$ can be written as a graph $\varphi_i$ of that plane. Without lost of generality, let the plane be the $xy$ plane in $\mb R^3$. By the computations in Theorem \ref{difference} and the rescaled equation (\ref{rescalesurfaceequation}), in $B_1$, $\tilde\Sigma_i$ satisfies a elliptic equation
\begin{equation}
\Delta\varphi_i=\dv(a_i\cdot\nabla\varphi_i)+b_i\cdot\nabla\varphi_i+c_i\varphi_i+d.
\end{equation}
Here every terms are defined on $\mb R^2\cap B_1(0)$. Again, when $i$ large, each terms on the right hand side goes to $0$. Then by implicit function theorem, if we fixed the normal direction to point upwards, we can solve $\varphi_{i,t}$ for boundary date $\varphi_{i,t}=\varphi_i+t$ on $\partial B_{1}(0)$. Then the graphs of $\varphi_{i,t}$ foliate a region of $B_1(0)\times\mb R^2$. Since we fixed the direction of normal vectors, we can apply maximal principle, which indicates that the leaf such that $\varphi_{i,t}(0)=0$ lies on one side of $\tilde\Sigma_i$. As a result, any sequence of dilations of $\Sigma$ must converge to the same limiting plane, which is just the tangent plane of that leaf at $0$.

Thus $\Sigma\cup\{0\}$ is a $C^1$ graph of a function $v$ in a neighbourhood of $0$. Since $v$ is a $C^{2,\alpha}$ solution to an elliptic equation except $0$, then $v$ is actually $C^{2,\alpha}$. Hence $\Sigma\cup\{0\}$ is a smooth disk.

We have already shown that $\Sigma\cup\{0\}$ is a union of smooth disks. So $\Sigma$ is an immersed surface, with locally finite many curvature concentration points. By maximal principle, at each touching point $\Sigma$ consists of two disks which are touching at that point. So $\Sigma$ is smoothly self-touching immersed.
\end{proof}

\begin{remark}
Suppose $y\in\mc S$, then from the proof we know in a small neighbourhood of $y$, $\Sigma$ is the union of two disks which are touching at $y$. By maximum principle, on each disks the normal vectors are pointing in opposite direction.
\end{remark}

\subsection{Smooth Convergence Besides Touching Points}
In this section we will follow \cite{colding2012smooth} to show the convergence is smooth besides touching points if the limit is not a self-shrinker. Note $\lambda$-surface has bounded mean curvature inside any fixed ball $B_R(x)$. Thus we can apply Allard's theorem in \cite{allard1972first} (also see \cite{choi1985space}, \cite{white1987curvature}, \cite{colding2012smooth}), and we only need to show the multiplicity must be one. 

From now on we will adapt one of the condition in previous section and assume the limit is smooth.

\begin{theorem}\label{multiplicity1}
The multiplicity of the convergence in Theorem \ref{convergence1} is one, if the limit is not a self-shrinker. 
\end{theorem}

\begin{proof}
Following Proposition 3.2 in \cite{colding2012smooth}, we will show that if the multiplicity is greater than one, then on $\Sigma$ there is a positive function $u$ such that $\mc L u=0$.

Since the convergence is smooth away from $\mc S$, we can choose $\eps_i\to 0$ and domain $\Omega_i\subset\Sigma$ exhausting $\Sigma\setminus\mc S$ so that each $\Sigma_i$ decomposes locally as a collection of graphs over $\Omega_i$, and is contained in the $\eps_i$-neighbourhood of $\Sigma$. Since $\Sigma_i$ is embedded (hence orientable), these sheets over $\Omega_i$ can be ordered by height. Let $w_i^+$ and $w_i^-$ be the functions representing the top sheet and the bottom sheet over $\Omega_i$, define $w_i=w_i^+-w_i^-$. By Theorem \ref{difference}
\[Lw_i=0\]
up to higher order correction terms.

Fixed $y\not\in\mc S$ and set $u_i=\frac{w_i}{w_i(y)}$. By embeddedness of $\Sigma_i$, $u_i$ is positive. Then Harnack inequality implies local $C^\alpha$ bound, and elliptic theory gives $C^{2,\alpha}$ bound. See \cite{colding2012smooth} and \cite{gilbarg2015elliptic} for the details of elliptic PDE. Then by Areazela-Ascoli theorem, a subsequence of $u_i$ converges uniformly in $C^2$ on a compact subset of $\Sigma\setminus\mc S$ to a non-negative function $u$ on $\Sigma\setminus\mc S$ such that
\[Lu=0, u(y)=1.\]

Next we show $u$ can extends smoothly across $\mc S$. This follows the standard removable singularity results for elliptic equations once we show that $u$ is bounded up to each $y_k$. Again we follow the idea of \cite{white1987curvature} and \cite{colding2012smooth} in Theorem \ref{limit}. Suppose $y_k\in\mc S$, from Theorem \ref{limit}, in a small neighbourhood of $y_k$ $\Sigma$ is union of at most two disks. Let $\tilde\Sigma$ be one of them. Suppose $w_i^+$ satisfies the elliptic equation
\[Lw_i^+=\dv(a_i\nabla w_i^+)+b_i\cdot\nabla w_i^++c_iw_i^++d\]
Then when $i$ large enough, by implicit function theorem, over the tangent plane of $\tilde\Sigma$, let's say $\mb R^2$, we can solve equation
\[Lv_t=\dv(a_i\nabla v_t)+b_i\cdot\nabla v_t+c_iv_t\]
with boundary value $v_t=t$ on $\partial B_\delta(y_k)$ for a fixed small $\delta$. Harnack inequality implies that $t/C\leq v_t\leq Ct$ for some $C>0$. Then by maximum principle $u_i$ is bounded by $Ct_1$ and $t_2/C$ if the boundary data of $u_i$ satisfies $t_2\leq u_i\leq t_1$. In conclusion, $u_i$ is bounded on $B_\delta(y_k)$ by a multiple of its supremum on $ B_\delta(y_k)\setminus B_{\delta/2}(y_k)$. This multiple is uniform because $u_i$ satisfies the same equation up to higher order. Then we conclude that $u$ is bounded in $B_{\delta}(y_k)$, hence $u$ has a removable singularity at each $y_k$, and thus extends to a non-negative solution of $Lu=0$ on all of $\Sigma$. Since $u(y)=1$, Harnack inequality implies that $u$ is positive everywhere.

However, this contradicts Theorem \ref{nopositivesol}. So the convergence can only has multiplicity one. Then we can extend the smoothly convergence across those singular points in $\mc S$ which are not touching points, i.e. has density $1$.
\end{proof}

So we extends our smooth convergence to whole $\Sigma$ besides local finite touching points with curvature concentration. The last step is to validate the terminology neck pinching we used in the main theorem. We claim that these touching points are really generated via a neck pinching process.

\begin{definition}
We say a touching point $p$ on $\Sigma$ is a {\bf neck pinching} point if for $0<r<r_0$, there are at most finitely many $\Sigma_i$ such that $\Sigma_i\cap B_r$ can be written as graphs over each connected disks of $\Sigma\cap B_r$.
\end{definition}

Intuitively, a neck pinching point is generated by splitting a connected neck into two different connected components.

\begin{proposition}
If $p\in\mc S$ is not a neck pinching point, then smooth convergence can be extended across $p$. 
\end{proposition}

\begin{proof}
If $p$ is not a neck pinching point, by definition locally we can write $\Sigma_i$ as graphs over each connected disks of $\Sigma\cap B_r$. Then we study the convergence of each such graphs, then we can argue the smooth convergence for each components just like non-touching points.
\end{proof}

Finally we discuss the situation that the limit is a self-shrinker. If the limit is a self-shrinker, and the convergence is of multiplicity larger than $2$, then we can find at least two sheets has the same orientation. Then we repeat the above discussion but just for the sheets with this orientation, we can construct a positive Jacobi field, which is a contradiction. Thus the multiplicity of the convergence is at most $2$.

Now we can conclude our main compactness theorem.

\begin{proof}[Proof of Theorem \ref{mainthm}]
Theorem \ref{convergence1} gives the convergence subsequence, Theorem \ref{limit} shows the limit is smooth and embedded besides touching points, and Theorem \ref{multiplicity1} shows the convergence is smooth besides those neck pinching points if the limit is not a self-shrinker. 
\end{proof}

\section{Rigidity of convex $\lambda$-Surfaces}
In this section we will prove a rigidity theorem. The main idea is that we can use the smoothly convergence $\lambda$-surfaces sequence to generate a solution to the linearized equation on the limit surface. Then the property of the solution would provide information of the surfaces which are closed to the limit.

\begin{theorem}\label{convexrigidity}
There exists $\delta_D<0$ such that for any $\lambda\in(\delta_D,+\infty)$, compact convex $\lambda$-surface with bounded diameter $D$ must be sphere $S^2_{r}(0)$ with $r=\sqrt{\lambda^2+4}-\lambda$.
\end{theorem}

The self-shrinkers case $\lambda=0$ was proved by Colding-Minicozzi in \cite{colding2012generic}. Later Heilman generalized Colding-Minicozzi's result to the case $\lambda\geq0$ in \cite{heilman2017symmetric}. Note they actually proved the rigidity for mean convex $\lambda$-surface. Their proof can not be generalized to $\lambda<0$ case because one key estimate in their proof relies on the positivity of $\lambda$.

In order to use the compactness theorem, first we need an area growth bound for convex surface. 

\begin{lemma}\label{convexarea}
Suppose $\Sigma$ is a convex surface in $\mb R^3$ without boundary. Then for any $x_0\in R^3$,
\begin{equation}
\area(B_R(x_0)\cap\Sigma)\leq 4\pi R^2.
\end{equation}
\end{lemma}

Note here our area bound holds for even non-compact convex surface.

\begin{proof}
A stronger argument appeared in the proof of the replacement lemma of Almgren-Simon in \cite{almgren1979existence}. See proof of Theorem 1 in section 3 in \cite{almgren1979existence}.

Since $\Sigma$ is convex, it divides the whole $\mb R^3$ into two parts: either the point in $\mb R^3$ lies in the convex hull of $\Sigma$, or it doesn't lie in the convex hull of $\Sigma$. Moreover, by convexity of $\Sigma$, for any point $x\in\mathbb R^3$ do not lie in $\Sigma$, we can find an unique points $\pi(x)$ on $\Sigma$ such that the distance between $\pi(x)$ and $x$ is the shortest distance from $x$ to $\Sigma$, and the straight line connect $x$ and $\pi(x)$ is perpendicular to $\Sigma$.

Now we define a vector field $V(x)$ out side $\Sigma$ such that $V(x)=\mathbf{n}(x)$ on $\Sigma$, and for any $x\in\mb R^3$ out side $\Sigma$, $V(x)=V(\pi(x))$. Since $\Sigma$ is convex, $\dv V\geq0$. This vector field is well defined because the projection is unique.

Let $S$ to be the part of $\partial B_{R}(x_0)$ lies out side $\Sigma$, and denote $D$ to be the domain of $B_R(x_0)$ lies out side of $\Sigma$. Then $\partial D=S\cup\Sigma$. Integration by part gives
\[\int_{\Sigma\cup S}V\cdot\mathbf n=\int_D\dv V\geq0\]
Note on $\Sigma$, $V\cdot\mathbf n=1$, where on $S$, $V\cdot\mathbf n\geq -1$. So
\[\area(\Sigma)\leq\area(S)\leq 4\pi R^2.\]
\end{proof}

Now we can prove the rigidity of bounded convex $\lambda$-surfaces.

\begin{proof}[Proof of Theorem \ref{convexrigidity}]
When $\lambda\geq0$ this rigidity is known in \cite{colding2012generic} and \cite{heilman2017symmetric}. So we only need to show $\delta_D<0$ exists. We argue by contradiction. Suppose such $\delta_D$ does not exists, then we can find a sequence of convex $\lambda_i$-surface $\Sigma_i$, $\lambda\to0$, such that $\Sigma_i$ is not a sphere $S^2_{\sqrt{\lambda_i^2+4}-\lambda_i}$. Convexity indicates that $\Sigma_i$ are all diffeomorphic to spheres, which gives the genus bound. By Lemma \ref{convexarea}, we get $\area(B_R(x_0)\cap\Sigma_i)\leq CR^2$ for a constant $C$. Thus we can apply our compactness Theorem \ref{mainthm} to $\{\Sigma_i\}$, and a subsequence must smoothly converge to a smooth compact self-shrinker $\Sigma$ because of the diameter bound. Convexity also implies that the the convergence must has multiplicity $1$. By maximum principle, $\Sigma$ can not have touching point, hence $\Sigma$ is embedded. Since the convergence is smooth and $\Sigma_i$ are all convex, $\Sigma$ is also convex. Then by \cite{colding2012generic} we know $\Sigma$ has to be the sphere $S^2_2(0)$.

On sphere with radius $r$
\[L=\Delta_{S^2_r}+\frac{1}{2}+\frac{2}{r^2}.\]
So when $r=2$, $L=\Delta_{S^2_2}+1$. 

Let us consider a family of maps between Banach spcaes:
\begin{equation}
f:\mathbb{R}\times C^{2,\alpha}\to C^{0,\alpha}
\end{equation}
where $f(t,u)=Lu+1-P_t(x,u,\nabla u,\nabla^2 u)$, such that $f(t,u)=0$ is the $\lambda$-surface equation with $\lambda=t$. The Fr\'echet differential of $f$ with respect to the second position at $(0,0)$ is just the linearized operator $L$. Since $1$ is not an eigenvalue of $S^2_2$, by Fredholm alternative theorem, $L$ is a local homomorphism from $C^{2,\alpha}\to C^{0,\alpha}$. So by implicit function theorem, locally there is a unique function $g:I\subset\mb R\to C^{2,\alpha}$ such that $f(t,g(t))=0$. Moreover, if we treat the sphere of radius $S^2_{\sqrt{t^2+4}-t}$ is a graph $u_t$ over $S^2_2$, then $u_t$ is a solution to $f(t,u_t)=0$. Thus they are the unique solutions to $f(t,u_t)=0$.

So when $i$ large enough, $\Sigma_i$ is sphere $S^2_{\sqrt{\lambda_i^2+4}-\lambda_i}$, which is a contradiction. Thus, there must be $\delta_0<0$ such that for $\lambda\in(0,+\infty)$.
\end{proof}

\begin{remark}
It is an interesting question to know more information about $\delta_D$. In the $1$ dimensional case Chang in \cite{chang20171} constructed examples to show that $\delta_D>-\infty$ for some $D$. We conjecture $\delta_D>-\infty$ also holds in $2$-dimensional $\lambda$-surfaces in $\mb R^3$ for some $D$.
\end{remark}

\bibliography{bibfile}
\bibliographystyle{alpha}

\end{document}